\documentclass[reqno]{amsart}
\usepackage{amsmath,amsthm,amsfonts}
\usepackage{url}

\newtheorem{thm}{Theorem}
\newtheorem{lem}{Lemma}

\theoremstyle{definition}

\newtheorem{remark}{Remark}

\DeclareMathOperator{\rank}{rank}

\newcommand{\Z}{\mathbb{Z}}

\title[Complex Hadamard matrices]{Complex Hadamard matrices attached to a $3$-class 
nonsymmetric association scheme}
\author{Takuya Ikuta}
\address{Kobe Gakuin University, Kobe, 650-8586, Japan}
\email{ikuta@law.kobegakuin.ac.jp}
\author{Akihiro Munemasa}
\address{Tohoku University, Sendai, 980-8579, Japan}
\email{munemasa@math.is.tohoku.ac.jp}
\thanks{This work was supported by JSPS KAKENHI grant number 17K05155.}

\date{April 19, 2019}
\keywords{association scheme, complex Hadamard matrix}
\subjclass[2010]{05E30,05B30}

\begin{document}

\begin{abstract}
In this paper we classify complex Hadamard matrices contained in the Bose--Mesner algebra of
nonsymmetric $3$-class association schemes.
As a consequence of our classification, 
we have two infinite families and some small examples of complex Hadamard matrices
contained in the Bose--Mesner algebra of a self-dual fission of a complete multipartite graph.
\end{abstract}
\maketitle

\section{Introduction}
A complex Hadamard matrix is a square matrix 
$W$ of order $n$ which satisfies $W\overline{W}^\top= nI$ and
all of whose entries are complex numbers of absolute value $1$.
A complex Hadamard matrix is said to be Butson-type, 
if all of its entries are roots of unity.
Complex Hadamard matrices which are not of Butson-type but 
are constructed combinatorially, found applications outside
combinatorics \cite{FS}. In particular, a family of complex Hadamard
matrices has been found in the Bose--Mesner algebras of 
nonsymmetric $2$-class association schemes
\cite{MW}.
In our earlier work \cite{MI},
we considered nonsymmetric hermitian complex Hadamard matrices 
belonging to the Bose--Mesner algebra of a nonsymmetric $3$-class
association scheme 
$\mathfrak{X}=(X,\{R_i\}_{i=0}^3)$,
where $R_1^\top=R_2$, $R_3$ symmetric.
Its first eigenmatrix is given by
\begin{equation}\label{P}
P=\begin{pmatrix}
1&\frac{k_1}{2}&\frac{k_1}{2}&k_2\\
1&\frac12(r+bi)&\frac12(r-bi)&-(r+1)\\
1&\frac12(r-bi)&\frac12(r+bi)&-(r+1)\\
1&\frac{s}{2}&\frac{s}{2}&-(s+1)
\end{pmatrix},
\end{equation}
where $k_1$ is an even positive integer, $r,s$ are integers, 
$b$ is a positive real number, and $i^2=-1$
(see Lemma~\ref{lem:song} for more precise information.).
Let $\mathfrak{A}$ be the Bose--Mesner algebra of $\mathfrak{X}$ 
which is the linear span of the adjacency matrices
$A_0,A_1,A_2,A_3$ of $\mathfrak{X}$, where $A_1^\top=A_2$, $A_3$ symmetric.
Let $w_1,w_2,w_3$ be complex numbers of absolute value $1$.
We assume that $w_1\not=w_2$, and
set
\begin{equation} \label{eq:W}
W=A_0+w_1A_1+w_2A_2+w_3A_3\in\mathfrak{A}.
\end{equation}
In \cite[Theorem 1]{MI}, we have shown the following: 
We assume that the matrix \eqref{eq:W} is a 
hermitian complex Ha\-da\-mard matrix and not a real Ha\-da\-mard matrix.
Then $\mathfrak{X}$ is a nonsymmetric association scheme 
whose unique nontrivial symmetric relation 
consists of $2\alpha$ cliques of size $2\alpha$,
and the first eigenmatrix of $\mathfrak{X}$ is given by 
\[
\begin{pmatrix}
1&\alpha(2\alpha-1)&\alpha(2\alpha-1)&2\alpha-1\\
1&\alpha i&-\alpha i&-1\\
1&-\alpha i&\alpha i&-1\\
1&-\alpha&-\alpha&2\alpha-1
\end{pmatrix},
\]
where $\alpha$ is a positive integer.
Moreover, $w_1=\pm i$ and $w_3=1$.
Note that we have assumed $w_1\neq w_2$ since, if $w_1=w_2$, then
$W$ belongs to the Bose--Mesner algebra of a strongly regular graph.
Such complex Hadamard matrices, or more generally, type-II matrices,
have been classified by Chan and Godsil \cite{CD}. Therefore, we
will assume $w_1\neq w_2$ throughout this paper.

As a natural problem, omitting the hermitian condition, 
we are interested in whether other complex Hadamard matrices arise or not.
In this paper, 
we show that such complex Hadamard matrices
are contained in the Bose--Mesner algebra of a self-dual fission of a complete multipartite graph.
More precisely, we have two infinite families and some small examples,
as given in the following main theorem.

\begin{thm}\label{thm:main}
With the above assumptions,
the matrix {\rm(\ref{eq:W})} is a complex Hadamard matrix 
if and only if $(k_1,k_2,r,s,b)=(2a(2a-1)c,2a-1,0,-2a,2a\sqrt{c})$
for some positive integers $a,c$, and
one of the following holds.
\begin{enumerate}
\item[{\rm(i)}] $c=1$, and
\begin{enumerate}
\item[{\rm(a)}] $(w_1,w_2,w_3)=(w,-w,1)$ with $|w|=1$, 
\item[{\rm(b)}] $(w_1,w_2,w_3)=(w_{\pm},w_{\mp},w_{\pm}w_{\mp})$,
where
\[
w_{\pm}=\frac{-(a-1)\mp a\sqrt{2a(a-1)}+(-a\pm(a-1)\sqrt{2a(a-1)})i}{2a^2-2a+1},
\]
\item[{\rm(c)}] $a=2$, $(w_1,w_2,w_3)= (\frac{3\pm4i}{5},-1,\frac{-3\mp4i}{5}),
(-1,\frac{3\pm4i}{5},\frac{-3\mp4i}{5})$,
\end{enumerate}
\item[{\rm(ii)}] $a=1$, $c=3$, and
\begin{enumerate}
\item[{\rm(d)}] $(w_1,w_2,w_3)=(\frac{1\pm2\sqrt{2}i}{3},-1,1),(-1,\frac{1\pm2\sqrt{2}i}{3},1)$,
\item[{\rm(e)}] $(w_1,w_2,w_3)=(\pm i,-1,\mp i),(-1,\pm i,\mp i)$.
\end{enumerate}
\end{enumerate}
\end{thm}

\begin{remark}
J{\o}rgensen \cite[Theorem 8]{J2} characterized association schemes
in Theorem~\ref{thm:main} (i) in terms of a Bush-type Hadamard matrix
of order $4a^2$ whose off-diagonal blocks are skew-symmetric.
Such a Bush-type Hadamard matrix can be constructed if there is a
Hadamard matrix of order $2a$ (see \cite[Corollary 5 (ii)]{Kh}).
Note that, if $a$ is a power of $2$, then
such an association scheme 
can also be constructed from a Galois ring of characteristic $4$ in \cite[Theorem 9]{IMY}.
An association scheme with the first eigenmatrix 
in Theorem~\ref{thm:main} (ii)
is given by 
as08[6] in \cite{H}.
\end{remark}

The organization of the paper is as follows.
In Section~\ref{sec:2},
we consider matrices which belong to the Bose--Mesner algebra of a commutative association scheme, and
we give a necessary and sufficient condition that such a matrix is a complex Hadamard matrix.
We also introduce a result of S.~Y.~Song (\cite{S})
which describes the eigenmatrices of nonsymmetric $3$-class association schemes.
In Section~\ref{sec:3}, we classify
complex Hadamard matrices attached to self-dual fissions of a complete multipartite graph.
In Section~\ref{sec:4}, we show that the Bose--Mesner algebra of
an association scheme in the remaining cases of Song's description
does not contain a complex Hadamard matrices.

All the computer calculations in this paper were performed with the help of Magma \cite{magma}.

\section{Association schemes and complex Hadamard matrices}\label{sec:2}
In this section,
we consider matrices
belonging to the Bose--Mesner algebra of a commutative association scheme.
Assuming that all
entries are complex numbers of absolute value $1$, 
we find conditions under which such a matrix is a complex Hadamard matrix.
We refer the reader to \cite{BI,BCN} for undefined terminology in the
theory of association schemes.

Let $X$ be a finite set with $n$ elements, and let
$(X,\{R_i\}_{i=0}^d)$ be a commutative association scheme 
with the first eigenmatrix 
$P=(P_{i,j})_{\substack{0\leq i\leq d\\ 0\leq j\leq d}}$.
We let $\mathfrak{A}$ denote the Bose--Mesner algebra
spanned by the adjacency matrices $A_0,A_1,\ldots,A_d$ of $(X,\{R_i\}_{i=0}^d)$.
Then the adjacency matrices are expressed as
\begin{equation}\label{eq:pij}
A_j=\sum_{k=0}^d P_{k,j}E_k \quad (j=0,1,\ldots,d),
\end{equation}
where $E_0=\frac{1}{n}J,E_1,\ldots,E_d$ are the primitive idempotents of $\mathfrak{A}$.
Set $m_k=\rank E_k$.
Let $w_0=1$ and $w_1,\ldots,w_d$ be complex numbers of absolute value $1$.
For each $j\in\{0,1,\dots,d\}$, define $j'$ by $A_{j'}=A_j^\top$.

Set
\begin{equation}\label{eq:WW}
W=\sum_{j=0}^d w_jA_j\in\mathfrak{A}.
\end{equation}
Let $X_0=1$ and let $X_j$ ($1\leq j\leq d$) be indeterminates.
For $k=0,1,\ldots,d$, let $e_k$ be the polynomial in $X_1,\dots,X_d$ defined by
\begin{equation}\label{ek}
e_k
=X_1\cdots X_d\left( \left(\sum_{j=0}^dP_{k,j}X_j\right)\left(\sum_{j=0}^d\frac{P_{k,j'}}{X_j}\right)-n\right).
\end{equation}

\begin{lem}\label{lem:1}
The following statements are equivalent.
\begin{itemize}
\item[\rm{(i)}] the matrix $W$ defined by {\rm(\ref{eq:WW})} is a complex Hadamard matrix,
\item[\rm{(ii)}] the sequence $(w_j)_{1\leq j\leq d}$ is a common zero of $e_k$ $(k=0,1,\ldots,d)$.
\end{itemize}
\end{lem}
\begin{proof}
By (\ref{eq:pij}), (\ref{eq:WW}) we have
\begin{align}
W&=\sum_{k=0}^d(\sum_{j=0}^dw_jP_{k,j})E_k,\label{eq:pijE}\\
\overline{W}^\top&=\sum_{j=0}^d\frac{1}{w_j}A_{j'} \qquad (A_{j'}=A_j^\top)\nonumber\\
&=\sum_{k=0}^d(\sum_{j=0}^d\frac{1}{w_j}P_{k,j'})E_k.\label{eq:tpijE}
\end{align}
By (\ref{eq:pijE}), (\ref{eq:tpijE}) we have
\[
W\overline{W}^\top
=\sum_{k=0}^d\left((\sum_{j=0}^dw_jP_{k,j})(\sum_{j=0}^d\frac{1}{w_j}P_{k,j'})\right)E_k.
\]
Therefore, the equivalence of (i) and (ii) follows.
\end{proof}

Let $\mathfrak{X}=(X,\{R_i\}_{i=0}^3)$ be a nonsymmetric $3$-class association scheme
with the first eigenmatrix (\ref{P}).
Consider the polynomial ring 
\begin{equation} \label{RC}
R=\mathbb{C}[X_1,X_2,X_3].
\end{equation}
Then by (\ref{ek}) we have the following.
\begin{align}
 e_0&=\left(1+\frac{k_1}{2}(X_1+X_2)+k_2X_3\right)
 \left(X_1X_2X_3+\frac{k_1}{2}(X_1+X_2)X_3+k_2X_1X_2\right)\nonumber\\
    &\quad -nX_1X_2X_3,\label{eq:0}\displaybreak[0]\\
 e_1&=\left(1+\frac{r+bi}{2}X_1+\frac{r-bi}{2}X_2-(r+1)X_3\right)\nonumber\\
 &\quad\times\left(X_1X_2X_3+(\frac{r+bi}{2}X_1+\frac{r-bi}{2}X_2)X_3-(r+1)X_1X_2\right) \nonumber\\
 &\quad-nX_1X_2X_3,\label{eq:1}\displaybreak[0]\\
 e_2&=\overline{e_1},\label{eq:2}\displaybreak[0]\\
 e_3&=\left(1+\frac{s}{2}(X_1+X_2)-(s+1)X_3\right)\nonumber\\
 &\quad\times\left(X_1X_2X_3+\frac{s}{2}(X_1+X_2)X_3-(s+1)X_1X_2\right)
   -nX_1X_2X_3.\label{eq:3}
\end{align}

\begin{lem}\label{lem:0430}
We have the following.
\begin{align}
4(e_0-e_3)&= \left( 2(k_1k_2+s^2+s)(X_1+X_2)+4(k_2+s+1)X_1X_2 \right)X_3^2 \nonumber\\
&\quad +\left( k_1^2(X_1+X_2)^2
     +2(k_1-s)(X_1X_2+1)(X_1+X_2) \right.\nonumber\\
&\qquad  \left.  -s^2(X_1^2+X_2^2)+2(2k_2^2-3s^2-4s-2)X_1X_2\right)X_3 \nonumber\\
&\quad +2( (k_1k_2+s(s+1))(X_1+X_2)+2(k_2+s+1) )X_1X_2, \label{eq:gen-e0e3}\displaybreak[0]\\
e_1-e_2&=-bi(X_1-X_2) 
\nonumber\\
&\qquad \times
\left( (r+1)(X_3^2+X_1X_2)-(X_1X_2+r(X_1+X_2)+1)X_3 \right).\label{eq:gen-e1e2}\displaybreak[0]
\end{align}
\end{lem}
\begin{proof}
Straightforward.
\end{proof}

\begin{lem}\label{lem:2}
Let $(w_1,w_2,w_3)$ be a common zero of the polynomials $e_k$ {\rm(}$k=0,1,2,3${\rm)}.
Then $(w_2,w_1,w_3)$ is also a common zero of the polynomials $e_k$.
\end{lem}
\begin{proof}
We have $e_j(X_2,X_1,X_3)=e_j(X_1,X_2,X_3)$ for $j=0,3$,
and $e_2(X_2,X_1,X_3)=e_1(X_1,X_2,X_3)$.
\end{proof}

S.~Y.~Song \cite[(5.3) Lemma]{S} showed the following.

\begin{lem}\label{lem:song}
Let $(1,m_1,m_1,m_2)$ denote the top row of the matrix $(1+k_1+k_2)P^{-1}$,
where $P$ is given by
\eqref{P}. Then one of the following holds.
\begin{itemize}
\item[\rm{(i)}] $(r,s,b^2)=(0,-(k_2+1),\frac{k_1(k_2+1)}{k_2})$, $m_1=\frac{(k_1+k_2+1)k_2}{2(k_2+1)}$,
\item[\rm{(ii)}] $(r,s,b^2)=(-(k_2+1),0,(k_2+1)(k_1+k_2+1))$, $m_1=\frac{k_1}{2(k_2+1)}$,
\item[\rm{(iii)}] $(r,s,b^2)=(-1,k_1,k_1+1)$, $m_1=\frac{(k_1+k_2+1)k_1}{k_1+1}$.
\end{itemize}
\end{lem}

In Lemma~\ref{lem:song},
(i) and (ii) are nonsymmetric fissions of a complete multipartite graph,
(i) is self-dual, and (ii) is non self-dual.
Also, (iii) is a nonsymmetric fission of a disjoint union of complete graphs.

\begin{lem}\label{lem:cmg}
Suppose that {\rm(i)} in {\rm Lemma~\ref{lem:song}} holds.
Define $a=-s/2$. 
Then
\begin{equation}\label{eq:k2a}\\
k_2=2a-1,
\end{equation}
and $a$ is a positive integer.
Moreover, there exists a positive integer $c$ such that
\begin{align}
k_1&=2a(2a-1)c,\label{eq:0312-1}\\
b&=2a\sqrt{c}.\label{eq:b-val}
\end{align}
Set
\begin{equation}\label{eq:0312-2}
L=(2a-1)c^2-2(a+1)c-1.
\end{equation}
Then we have the following.
\begin{enumerate}
\item[\rm{(i)}] $L\not=0$,
\item[\rm{(ii)}] Suppose that $L<0$.
Then one of the following holds.
 \begin{enumerate}
 \item[\rm{(a)}] $c=1$,
 \item[\rm{(b)}] $(a,c)=(1,3)$, that is, $k_1=6$, $k_2=1$, $b=2\sqrt{3}$.
 \end{enumerate}
\end{enumerate}
\end{lem}
\begin{proof}
Since $-a$ is an eigenvalue of the adjacency matrix $A_1$,
$a$ is an integer.
Since $s=-(k_2+1)$, we have (\ref{eq:k2a}).
The symmetrization is the 
complete multipartite graph of the part size $k_2+1$.
Thus, $k_2+1$ is a divisor of $n=1+k_1+k_2$, and hence
$k_2+1$ is a divisor of $k_1$.
Since $b^2=k_1(k_2+1)/k_2$ is an integer, $k_2$ is a divisor of $k_1$.
Hence $k_2(k_2+1)$ is a divisor of $k_1$.
Thus, there exists a positive
integer $c$ such that (\ref{eq:0312-1}) holds.
Since $b$ is a positive real number, we have (\ref{eq:b-val}).

(i) Suppose that $L=0$.
We regard (\ref{eq:0312-2}) as a quadratic of $c$.
Since the discriminant is $a(a+4)$,
there exists a positive integer $m$ such that $a(a+4)=m^2$.
Then $(a+2+m)(a+2-m)=4$.
This contradicts a positive integer $a$.

(ii) Since $L<0$, we have
\[
1\leq c\leq\frac{a+1+\sqrt{a(a+4)}}{2a-1}.
\]
Then we have the following:
If $a=1$ then $c\in\{1,2,3,4\}$. If $a=2$ then $c\in\{1,2\}$. If $a\geq3$ then $c=1$.
For $(a,c)=(1,2), (1,4), (2,2)$ we have $m_1=\frac32, \frac52, \frac{21}{2}$ by (i) in Lemma~\ref{lem:song}, respectively.
Hence we have the assertions (a) or (b).
\end{proof}

To conclude this section,
we note that the matrices
described in Theorem~\ref{thm:main} are indeed complex Hadamard
matrices.
By Lemma~\ref{lem:1}, it suffices to show that
$(w_1,w_2,w_3)$ is common zero of the polynomials (\ref{eq:0})--(\ref{eq:3}), and
it is easy to do this.

\section{Self-dual fissions of a complete multipartite graph} \label{sec:3}

Throughout this section,
we suppose that case (i) in Lemma~\ref{lem:song} holds,
and that 
the matrix (\ref{eq:W}) is a complex Hadamard matrix with $w_1\not=w_2$.
By Lemma~\ref{lem:1}, $(w_1,w_2,w_3)$ is a common zero of the polynomials $e_k$ ($k=0,1,2,3$).
\begin{lem}\label{lem:0426}
We have the following.
\begin{itemize}
\item[{\rm(i)}] $w_3=1$ or $w_3=w_1w_2$.
\item[{\rm(ii)}] $w_2=-w_1$ or $g(w_1,w_2)=0$,
where
\begin{align}
g(X_1,X_2)&=
2(X_1X_2+1)+((2a-1)c-1)(X_1+X_2).
\label{eq:g}
\end{align}
\item[{\rm(iii)}]
If $w_1w_2=1$, then $w_2=-w_1=\pm i$.
\end{itemize}
\end{lem}
\begin{proof}
After specializing $r=0$, $s=-2a$, $k_1=2a(2a-1)c$,
and $k_2=2a-1$ in (\ref{eq:gen-e0e3}) and (\ref{eq:gen-e1e2}),
we have 
\begin{align}
e_1-e_2&=bi(X_1-X_2)(X_3-1)(X_1X_2-X_3), \label{eq:e1e2} \\
e_0-e_3&=a((2a-1)c+1)(X_1+X_2)g_1(X_1,X_2,X_3), \label{eq:e0e3}
\end{align}
where
\begin{align}
g_1(X_1,X_2,X_3)&=(2a-1)(X_1X_2+X_3^2) \nonumber\\
&\quad +\left(X_1X_2+1+a((2a-1)c+1)(X_1+X_2)\right)X_3. \label{eq:g1}
\end{align}
From (\ref{eq:e1e2}) we have (i).

Observe, by (\ref{eq:g}) and (\ref{eq:g1}),
\begin{align}
X_1X_2g_1(X_1,X_2,1)&=g_1(X_1,X_2,X_1X_2)=aX_1X_2g(X_1,X_2).
\label{eq:ag1}
\end{align}
From (\ref{eq:e0e3}) we have $w_2=-w_1$ or $g_1(w_1,w_2,w_3)=0$.
Suppose that $g_1(w_1,w_2,w_3)=0$. Since $w_3=1$ or $w_3=w_1w_2$ by (i),
we have $g(w_1,w_2)=0$ by (\ref{eq:ag1}).

Finally, we suppose that $w_1w_2=1$.
Then $w_3=1$ by (i).
Suppose that $g(w_1,\frac{1}{w_1})=0$.
Then by (\ref{eq:g}) we have
\begin{equation}\label{eq:0520-1}
((2a-1)c-1)w_1^2+4w_1+(2a-1)c-1=0.
\end{equation}
Since $e_1(w_1,\frac{1}{w_1},1)=0$,
we have
\begin{equation}\label{eq:0520-2}
acw_1^4+2a((a-1)c+1)w_1^2+ac=0.
\end{equation}
From (\ref{eq:0520-1}) and (\ref{eq:0520-2}), by using the notation of (\ref{eq:0312-2}),
we have
\[
( (2a-1)c+1 )L=0.
\]
Since $a$ and $c$ are positive integers,
this contradicts Lemma~\ref{lem:cmg} (i).
By (ii) we have $w_2=-w_1$,
that is,
$w_1=\pm i$.
Therefore we have (iii).
\end{proof}

\begin{lem}\label{lem:7}
Suppose that $w_2=-w_1$.
Then we have {\rm (a)} in {\rm Theorem~\ref{thm:main}} {\rm(i)}.
\end{lem}
\begin{proof}
By (i) in Lemma~\ref{lem:0426} we have $w_3=1$ or $w_3=w_1w_2$.

First suppose that $w_3=1$.
After specializing $X_2=-X_1$ and $X_3=1$, we have
\[
e_1=-2a(c-1)X_1^2.
\]
Hence $c=1$.
Therefore we have {\rm (a)} in {\rm Theorem~\ref{thm:main}} {\rm(i)}.

Secondly suppose that $w_3\not=1$.
Then $w_3=w_1w_2$,
that is, $w_3=-w_1^2$.
After specializing $X_2=-X_1$ and $X_3=-X_1^2$, we have
\[
e_1=X_1^2(X_1^4+2((c-1)a+1)X_1^2+1).
\]
Hence $w_3^2-2((c-1)a+1)w_3+1=0$.
Since $a$ and $c$ are positive integers,
we have $w_3=1$.
This contradicts $w_3\not=1$.
\end{proof}


\begin{lem}\label{lem:0518}
Suppose that $w_2\not=-w_1$ and either $w_1$ or $w_2$ is $-1$.
Then $(a,c)=(2,1)$ or $(a,c)=(1,3)$.
\end{lem}
\begin{proof}
We may set $w_2=-1$ without loss of generality.
Since $w_2\not=-w_1$, by (ii) in Lemma~\ref{lem:0426} 
we have $g(w_1,w_2)=0$.
Then by $w_2=-1$ we have $((2a-1)c-3)(w_1-1)=0$.
By $w_1\not=1$ we have $(2a-1)c=3$.
Since $a$ and $c$ are positive integers,
we have the assertion.
\end{proof}

\begin{lem}\label{lem:9}
With the assumption in {\rm Lemma~\ref{lem:0518}}, 
suppose that $(a,c)=(2,1)$.
Then we have {\rm (c)} in {\rm Theorem~\ref{thm:main}} {\rm(i)}.
\end{lem}
\begin{proof}
We may set $w_2=-1$ without of loss of generality.
By (\ref{eq:k2a})--(\ref{eq:b-val}) we have $k_1=12$, $k_2=3$, and $b=4$.
By (i) in Lemma~\ref{lem:0426} we have $w_3=1$ or $w_3=w_1w_2$.

First suppose that $w_3=1$.
After specializing $X_2=-1$, $X_3=1$, $k_1=12$, and $k_2=3$, we have
\[
e_1=-4(X_1-1)^2.
\]
Hence $w_1=1$.
This contradicts $w_1\not=1$.

Secondly suppose that $w_3\not=1$.
By (i) in Lemma~\ref{lem:0426} we have $w_3=w_1w_2$, that is, $w_3=-w_1$.
After specializing $X_2=-1$, $X_3=-X_1$, $k_1=12$, and $k_2=3$, we have
\[
e_1=X_1(5X_1^2-6X_1+5).
\]
Hence $5w_1^2-6w_1+5=0$, that is, $w_1=(3\pm4i)/5$.
Together with the use of Lemma~\ref{lem:2}, we have 
{\rm (c)} in {\rm Theorem~\ref{thm:main}} {\rm(i)}.
\end{proof}

\begin{lem}\label{lem:10}
With the assumption in {\rm Lemma~\ref{lem:0518}}, 
suppose that $(a,c)=(1,3)$.
Then we have  {\rm (d)} or {\rm (e)} in {\rm Theorem~\ref{thm:main}} {\rm(ii)}.
\end{lem}
\begin{proof}
We may set $w_2=-1$ without of loss of generality.
By (\ref{eq:k2a})--(\ref{eq:b-val}) we have $k_1=6$, $k_2=1$, and $b=2\sqrt{3}$.
By (i) in Lemma~\ref{lem:0426} we have $w_3=1$ or $w_3=w_1w_2$.

First suppose that $w_3=1$.
After specializing $X_2=-1$, $X_3=1$, $k_1=6$, and $k_2=1$, we have
\[
e_1=-(3X_1^2-2X_1+3).
\]
Hence $3w_1^2-2w_1+3=0$, that is, $w_1=(1\pm2\sqrt{2}i)/3$.
Together with the use of Lemma~\ref{lem:2}, we have 
{\rm (d)} in {\rm Theorem~\ref{thm:main}} {\rm(ii)}.

Secondly suppose that 
$w_3=w_1w_2$, that is, $w_3=-w_1$.
After specializing $X_2=-1$, $X_3=-X_1$, $k_1=6$, and $k_2=1$, we have
\[
e_1=4X_1(X_1^2+1).
\]
Hence $w_1=\pm i$.
Together with the use of Lemma~\ref{lem:2}, we have 
{\rm (e)} in {\rm Theorem~\ref{thm:main}} {\rm(ii)}.
\end{proof}

\begin{lem}\label{lem:11}
Suppose that $w_2\not=-w_1$ and $w_1,w_2\not=\pm1$.
Then we have {\rm (b)} in {\rm Theorem~\ref{thm:main}} {\rm(i)}.
\end{lem}
\begin{proof}
By (iii) in Lemma~\ref{lem:0426} we have $w_1w_2\neq1$, while
$w_1\neq w_2$ by our assumption. 
This means that $w_1$, $\overline{w_1}$, $w_2$, $\overline{w_2}$ are mutually distinct.
Let $x_j=w_j+1/w_j$ for $j=1,2$.
Then the polynomial having $w_1$, $\overline{w_1}$, $w_2$, $\overline{w_2}$ as roots is given by
\begin{equation}\label{eq:ff1}
f(X)=(X^2-x_1X+1)(X^2-x_2X+1).
\end{equation}
By the assumption $w_2\neq-w_1$ and (ii) of Lemma~\ref{lem:0426},
we have $g(w_1,w_2)=0$.
Let $\mathcal{J}$ denote the ideal of (\ref{RC}) 
generated by the polynomials $e_k$ ($k=0,1,2,3$) and (\ref{eq:g}).
By (i) in Lemma~\ref{lem:0426} we have $w_3=1$ or $w_3=w_1w_2$.

First suppose that $w_3=1$. 
Let $\mathcal{J}_1$ denote the ideal of (\ref{RC}) 
generated by $\mathcal{J}$ and $X_3-1$.
Then $(w_1,w_2,1)$ is a common zero of polynomials in $\mathcal{J}_1$.
We can verify that $\mathcal{J}_1$ contains $f_1(X_1)$, $f_1(X_2)$, where
\begin{align}
f_1(X)&=a_0X^4+a_1X^3+a_2X^2+a_1X+a_0,\label{eq:f1}\\
a_0&=2ac, \nonumber\\
a_1&=2(c-1)((2a-1)c-1), \nonumber\\
a_2&=4ac^2((c-1)a-c)+(c-1)(c^2+2c+5).\nonumber
\end{align}
Then $f_1(w_1)=f_1(w_2)=0$.
Since $f_1(X)$ has real coefficients, we have $f_1(X)=a_0f(X)$.
Comparing the coefficients of (\ref{eq:ff1}) and (\ref{eq:f1}),
we have
\begin{align}
a_0(x_1+x_2)+a_1&=0,\label{eq:100}\\
a_0x_1x_2+2a_0-a_2&=0.\label{eq:101}
\end{align}
By (\ref{eq:100}) and (\ref{eq:101}),
$x_1$ and $x_2$ are the roots of
\begin{equation}\label{eq:disc1}
a_0X^2+a_1X-2a_0+a_2=0.
\end{equation}
Then by using the notation of (\ref{eq:0312-2}), since the discriminant of (\ref{eq:disc1}) is
\[
-((2a-1)c+1)^2L
\]
and $x_1\not=x_2$, we have $L<0$.
First we consider (a) of Lemma~\ref{lem:cmg} (ii).
Then by (\ref{eq:disc1}) we have $X^2-4=0$, that is, $w_1,w_2\in\{\pm1\}$.
This contradicts our assumption.
Secondly we consider (b) in Lemma~\ref{lem:cmg} (ii).
Then, (\ref{eq:100}) and (\ref{eq:101}) are 
$3(x_1+x_2)+4=0$ and $3x_1x_2+4=0$, respectively.
Then $\{x_1,x_2\}=\{-2,\frac23\}$, that is,
either $w_1$ or $w_2$ is $-1$.
This contradicts our assumption.

Secondly suppose that 
$w_3=w_1w_2\not=1$.
Let $\mathcal{J}_2$ denote the ideal of (\ref{RC}) 
generated by $\mathcal{J}$ and $X_3-X_1X_2$.
Then $(w_1,w_2,w_1w_2)$ is a common zero of the polynomials in $\mathcal{J}_2$.
We can verify that $\mathcal{J}_2$ contains
$X_i\left(((2a-1)c-1)X_i+2\right)f_2(X_i)$, $i=1,2$, where
\begin{align}
f_2(X)&=b_0X^4+b_1X^3+b_2X^2+b_1X+b_0,\label{eq:f2}\\
b_0&=(c+1)(4ac(a-1)+c+1),\nonumber\\
b_1&=4(a(c-1)+2)((2a-1)c-1),\nonumber\\
b_2&=2a(2a-1)^2c^3-2(2a-1)(2a^2-a+1)c^2-2(a-2)c-2(5a-9).\nonumber
\end{align}
Then $f_2(w_1)=f_2(w_2)=0$.
Since $f_2(X)$ has real coefficients, we have $f_2(X)=b_0f(X)$.
Comparing coefficients of (\ref{eq:ff1}) and (\ref{eq:f2}),
we have
\begin{align}
&b_0(x_1+x_2)+b_1=0,\label{eq:102}\\
&b_0x_1x_2+2b_0-b_2=0.\label{eq:103}
\end{align}
By (\ref{eq:102}) and (\ref{eq:103})
$x_1$ and $x_2$ are the roots of
\begin{align}
&b_0X^2+b_1X-2b_0+b_2=0.\label{eq:disc2}
\end{align}
Then by using the notation of (\ref{eq:0312-2}), the discriminant of (\ref{eq:disc2}) is
\[
-2a((2a-1)c+1)^2((2a-1)c+2a-3)L.
\]
Since $x_1\not=x_2$, we have $(2a-1)c+2a-3\not=0$ and $L<0$.

First we consider (a) in Lemma~\ref{lem:cmg} (ii).
Then by (\ref{eq:disc2}) the real parts of $w_1$ and $w_2$ are given by
\[
\frac{-a+1\pm a\sqrt{2a(a-1)} }{2a^2-2a+1}.
\]
Since $|w_1|=|w_2|=1$, we have (b) in Theorem~\ref{thm:main} (i).

Secondly we consider (b) in Lemma~\ref{lem:cmg} (ii).
Then by $a=1$ and $c=3$, (\ref{eq:disc2}) is $X(X+2)=0$, that is,
either $w_1$ or $w_2$ is $-1$.
This contradicts our assumption.
\end{proof}

Therefore, we have shown that, in case (i) in Lemma~\ref{lem:song},
complex Hadamard matrices must arise in one of the ways described
in Theorem~\ref{thm:main}. In the next section, we will show 
the nonexistence of a complex Hadamard
matrices for cases (ii) and (iii) in Lemma~\ref{lem:song}.

\section{Other cases}\label{sec:4}
In this section we consider cases (ii) and (iii) in Lemma~\ref{lem:song}.
\begin{lem}\label{lem:case2}
Suppose that the matrix {\rm(\ref{P})} satisfies case {\rm(ii)} in {\rm Lemma~\ref{lem:song}}.
Then there does not exist complex Hadamard matrices of the form {\rm(\ref{eq:W})}.
\end{lem}
\begin{proof}
After specializing $r=-(k_2+1)$ and $s=0$, we have
\[
-e_3=X_1X_2(X_3^2+(k_1+k_2-1)X_3+1).
\]
Hence
\begin{equation}\label{eq:case2}
w_3^2+(k_1+k_2-1)w_3+1=0.
\end{equation}
Since $k_1+k_2-1\geq2$,
we have $w_3=-1$ and $k_1+k_2=3$.
Since $k_1$ is even, we have $k_1=2$ and $k_2=1$.
Then by (ii) in Lemma~\ref{lem:song} we have $m_1=\frac12\not\in\Z$.
This is a contradiction.
\end{proof}

\begin{lem}\label{lem:case3}
Assume that the matrix {\rm(\ref{P})} satisfies case {\rm(iii)} in {\rm Lemma~\ref{lem:song}}.
Then there does not exist complex Hadamard matrices of the form {\rm(\ref{eq:W})}.
\end{lem}
\begin{proof}
After specializing $r=-1$ and $s=k_1$ in (\ref{eq:gen-e0e3}) and (\ref{eq:gen-e1e2}),
we have 
\begin{align}
e_0-e_3=&\frac12(k_1+k_2+1)h_1(X_1,X_2,X_3), \label{e0e3}\\
e_1-e_2=&biX_3(X_1-X_2)(X_1-1)(X_2-1), \label{e1e2}
\end{align}
where
\begin{align*}
h_1(X_1,X_2,X_3)&=k_1(X_1+X_2)(X_1X_2+X_3^2)-2(k_1-k_2+1)X_1X_2X_3 \\
&\quad+2X_1X_2(X_3^2+1).
\end{align*}
Then by (\ref{e0e3}) we have
\begin{equation}\label{eq:g123}
h_1(w_1,w_2,w_3)=0.
\end{equation}
We may set $w_1=1$ without of loss of generality by (\ref{e1e2}).
After specializing $X_1=1$, $r=-1$, and $s=k_1$, we have
\[
-4e_1=X_3((k_1+2)X_2^2+2(k_1+2k_2)X_2+k_1+2).
\]
Hence
\begin{equation}\label{eq:g123e1}
w_2^2+\frac{2(k_1+2k_2)}{k_1+2}w_2+1=0.
\end{equation}
Since $\frac{2(k_1+2k_2)}{k_1+2}\geq2$, we have $w_2=-1$.
Then by (\ref{eq:g123e1}) we have $k_2=1$.
Then by (iii) in Lemma~\ref{lem:song} we have $m_1=\frac{(k_1+2)k_1}{k_1+1}\not\in\Z$.
This is a contradiction.
\end{proof}


\appendix
\section{Verification by Magma}

\begin{verbatim}
QQ<i>:=QuadraticField(-1);
PR<w1,w2,w3,k1,k2,r,s,b,c>:=PolynomialRing(QQ,9);
F:=FieldOfFractions(PR);

d:=3;
eigenP:=Matrix(PR,4,4,[
 1,k1/2,k1/2,k2,
 1,1/2*(r+b*i),1/2*(r-b*i),-(r+1),
 1,1/2*(r-b*i),1/2*(r+b*i),-(r+1),
 1,s/2,s/2,-(s+1)]);
P:=func<i,j|eigenP[i+1,j+1]>;

conjugateP:=Matrix(PR,4,4,[
 1,k1/2,k1/2,k2,
 1,1/2*(r-b*i),1/2*(r+b*i),-(r+1),
 1,1/2*(r+b*i),1/2*(r-b*i),-(r+1),
 1,s/2,s/2,-(s+1)]);
P1:=func<i,j|conjugateP[i+1,j+1]>;

n:=&+[P(0,j) : j in [0..d]];
n eq 1+k1+k2;
WW:=[1,w1,w2,w3];
W:=func< j | WW[j+1] >;

e:=func< k | &*[ x : x in WW ]*( 
             ( &+[P(k,j)*W(j):j in [0..d]] )
			 *( &+[P1(k,j)/W(j):j in [0..d]] )-n ) >; // (5)
ee:=[e(0),e(1),e(2),e(3)];

// Eqs (9), (10), (11), (12).
eq9:=(1+k1/2*(w1+w2)+k2*w3)*(w1*w2*w3+k1/2*(w1+w2)*w3+k2*w1*w2)
 -n*w1*w2*w3;
eq10:=(1+(r+b*i)/2*w1+(r-b*i)/2*w2-(r+1)*w3)
 *(w1*w2*w3+((r+b*i)/2*w1+(r-b*i)/2*w2)*w3-(r+1)*w1*w2)-n*w1*w2*w3;
eq11:=(1+(r-b*i)/2*w1+(r+b*i)/2*w2-(r+1)*w3)
 *(w1*w2*w3+((r-b*i)/2*w1+(r+b*i)/2*w2)*w3-(r+1)*w1*w2)-n*w1*w2*w3;
eq12:=(1+s*(w1+w2)/2-(s+1)*w3)*(w1*w2*w3+s*(w1+w2)/2*w3-(s+1)*w1*w2)
 -n*w1*w2*w3;
ee eq [eq9,eq10,eq11,eq12];


// Lemma 2
c2:=2*(k1*k2+s^2+s)*(w1+w2)+4*(k2+s+1)*w1*w2;
c1:=k1^2*(w1+w2)^2
     +2*(k1-s)*(w1*w2+1)*(w1+w2)
     -s^2*(w1^2+w2^2)+2*(2*k2^2-3*s^2-4*s-2)*w1*w2;
c0:=2*((k1*k2+s*(s+1))*(w1+w2)+2*(k2+s+1))*w1*w2;
4*(e(0)-e(3)) eq c2*w3^2+c1*w3+c0; // (13)
e(1)-e(2) eq 
 -b*i*(w1-w2)*((r+1)*(w3^2+w1*w2)-(w1*w2+r*(w1+w2)+1)*w3); // (14)


// Lemma 3
[ ff @ hom< PR-> PR | [w2,w1,w3,k1,k2,r,s,b,c] > : ff in ee ]
 eq [e(0),e(2),e(1),e(3)];


// Lemma 4 (i)
eqOfLem4:=b^2*k2-k1*(k2+1);


// Lemma 5
a:=-s/2;
K2:=2*a-1; // (15)
eq15:=K2-k2;
K1:=2*a*(2*a-1)*c; // (16)
eq16:=K1-k1;
eq17:=b^2-4*a^2*c; // (17)
L:=(2*a-1)*c^2-2*(a+1)*c-1; // (18)
hom3:=hom< PR->PR | [w1,w2,w3,K1,K2,0,s,b,c] >;
e0:=e(0) @ hom3;
e1:=e(1) @ hom3;
e2:=e(2) @ hom3;
e3:=e(3) @ hom3;
eqOfLem4i:=eqOfLem4 @ hom3;
I:=ideal< PR | e0,e1,e2,e3,eqOfLem4i,eq15,eq16,eq17 >;


// Lemma 5
Lem5_1:=function(a,c)
  k1:=2*a*(2*a-1)*c;
  k2:=2*a-1;
  m1:=(k1+k2+1)*k2/(2*(k2+1));
  return m1;
end function;
 Lem5_1(1,2) eq 3/2;
 Lem5_1(1,4) eq 5/2;
 Lem5_1(2,2) eq 21/2;


//// Section 3

// Lemma 6
g1:=(2*a-1)*(w1*w2+w3^2)
      +(w1*w2+a*((2*a-1)*c-1)*(w1+w2)+1)*w3; // (22)
e1-e2 eq b*i*(w1-w2)*(w3-1)*(w1*w2-w3); // (20)
e0-e3 eq a*((2*a-1)*c+1)*(w1+w2)*g1; // (21)
 
g:=2*(w1*w2+1)+((2*a-1)*c-1)*(w1+w2); // (19)
a*w1*w2*g eq w1*w2*g1 @ hom< PR->PR | [w1,w2,1,K1,K2,0,s,b,c] >;
a*w1*w2*g eq g1 @ hom< PR->PR | [w1,w2,w1*w2,K1,K2,0,s,b,c] >;

eq24:=((2*a-1)*c-1)*w1^2+4*w1+(2*a-1)*c-1; // (24)
w1*(g @ hom< F->F | [w1,(F!w1)^(-1),1,K1,K2,0,s,b,c] >) eq eq24;

eq25:=a*c*w1^4+2*((a-1)*c+1)*w1^2+a*c; // (25)
4*w1^2*(e1 @ hom< F->F | [w1,1/w1,1,K1,K2,r,s,b,c ] >)
 eq 2*s*eq25 - eq17*(w1^2-1)^2;

cL:=2*((2*a-1)*c+1)*L;
c24:=4*a*c*w1^3-a*c*(2*a*c-c-1)*w1^2+(8*a*c-8*c+8)*w1
 -(2*a*c-c-1)*(a*c-2*c+2);
c25:=-((8*a*c-4*c-4)*w1-(2*a*c-c+3)*(2*a*c-c-5));
cL eq c24*eq24+c25*eq25;


// Lemma 7
-2*a*(c-1)*w1^2
 eq e1 @ hom< PR->PR | [w1,-w1,1,K1,K2,r,s,b,c ] > + w1^2*eq17;
w1^2*(w1^4+2*((c-1)*a+1)*w1^2+1)
 eq e1 @ hom< PR->PR | [w1,-w1,-w1^2,K1,K2,r,s,b,c ] > - w1^4*eq17;


// Lemma 8
((2*a-1)*c-3)*(w1-1)
 eq g @ hom< PR->PR | [w1,-1,w3,K1,K2,r,s,b,c ] >;


// Lemma 9, (a,c)=(2,1), that is, (s,c)=(-4,1)
12 eq K1 @ hom< PR->PR | [w1,-1,w3,K1,K2,r,-4,4,1 ] >;
3 eq K2 @ hom< PR->PR | [w1,-1,w3,K1,K2,r,-4,4,1 ] >;
-4*(w1-1)^2 eq e1 @ hom< PR->PR | [w1,-1,1,K1,K2,r,-4,4,1 ] >;
w1*(5*w1^2-6*w1+5) eq 
 e1 @ hom< PR->PR | [w1,-1,-w1,K1,K2,r,-4,4,1 ] >;


// Lemma 10, (a,c)=(1,3), that is, (s,c)=(-2,3)
6 eq K1 @ hom< PR->PR | [w1,-1,w3,K1,K2,r,-2,b,3 ] >;
1 eq K2 @ hom< PR->PR | [w1,-1,w3,K1,K2,r,-2,b,3 ] >;
eqOfLem101:=-(3*w1^2-2*w1+3);
eqOfLem101 -1/4*(b^2-12)*(w1+1)^2
 eq e1 @ hom< PR->PR | [w1,-1,1,K1,K2,r,-2,b,3 ] >;
eqOfLem102:=4*w1*(w1^2+1);
eqOfLem102 +1/4*(b^2-12)*w1*(w1+1)^2
 eq e1 @ hom< PR->PR | [w1,-1,-w1,K1,K2,r,-2,b,3 ] >;


// Lemma 11
// w3=1:
I1:=ideal< PR | I,w3-1,g >;
a0:=2*a*c;
a1:=2*(c-1)*((2*a-1)*c-1);
a2:=4*a*c^2*((c-1)*a-c)+(c-1)*(c^2+2*c+5);
f1w1:=a0*w1^4+a1*w1^3+a2*w1^2+a1*w1+a0;
f1w2:=a0*w2^4+a1*w2^3+a2*w2^2+a1*w2+a0;
s*f1w1 in I1;
s*f1w2 in I1;
dicsOfeq30:=a1^2-4*a0*(-2*a0+a2);
dicsOfeq30 eq -4*((2*a-1)*c+1)^2*L;

// w3=w1*w2:
I2:=ideal< PR | I,w1*w2-w3,g >;
b0:=(c+1)*(4*a*c*(a-1)+c+1);
b1:=4*(a*(c-1)+2)*((2*a-1)*c-1);
b2:=2*a*(2*a-1)^2*c^3-2*(2*a-1)*(2*a^2-a+1)*c^2-2*(a-2)*c-2*(5*a-9);
f2w1:=b0*w1^4+b1*w1^3+b2*w1^2+b1*w1+b0;
f2w2:=b0*w2^4+b1*w2^3+b2*w2^2+b1*w2+b0;
w1*(w1*((2*a-1)*c-1)+2)*f2w1 in I2;
w2*(w2*((2*a-1)*c-1)+2)*f2w2 in I2;
dicsOfeq34:=b1^2-4*b0*(-2*b0+b2);
dicsOfeq34 eq -8*a*((2*a-1)*c+1)^2*((2*a-1)*c+2*a-3)*L;


//// Section 4

// Lemma 12
-e(3) @ hom< PR->PR | [w1,w2,w3,k1,k2,-(k2+1),0,b,c] >
 eq w1*w2*(w3^2+(k1+k2-1)*w3+1);


// Lemma 13
hom3c:=hom< PR->PR | [w1,w2,w3,k1,k2,-1,k1,b,c] >;
eqL13:=[ e @ hom3c : e in ee ];
h1:=k1*(w1+w2)*(w1*w2+w3^2)-2*(k1-k2+1)*w1*w2*w3+2*w1*w2*(w3^2+1);
eqL13[1]-eqL13[4] eq 1/2*(k1+k2+1)*h1; // (36)
eqL13[2]-eqL13[3] eq b*i*w3*(w1-w2)*(w1-1)*(w2-1); // (37)
eqOfLem13:=w3*((k1+2)*w2^2+2*(k1+2*k2)*w2+k1+2);
subs13:=hom< PR->PR | [1,w2,w3,b^2-1,k2,-1,b^2-1,b,c] >;
-4*eqL13[2] @ subs13 eq eqOfLem13 @ subs13;

// Total time: 1.229 seconds, Total memory usage: 32.09MB
\end{verbatim}

\end{document}